\documentclass[10pt,journal,compsoc]{IEEEtran}

\usepackage{amsmath}
\usepackage{amssymb}
\usepackage[dvipdfmx]{graphicx}
\usepackage{color}
\usepackage{dsfont}

\usepackage{supertabular}
\usepackage{comment}

\ifCLASSOPTIONcompsoc

\usepackage{cite}
\usepackage{url}

\usepackage{theorem}
\theorembodyfont{\rmfamily}

\newtheorem{definition}{Definition}

\newtheorem{theorem}{Theorem}
\newtheorem{lemma}{Lemma}

\newcommand{\bm}[1]{\boldsymbol{#1}}

\DeclareMathOperator{\sfsin}{\mathsf{sin}}
\DeclareMathOperator{\sfcos}{\mathsf{cos}}
\DeclareMathOperator{\sftan}{\mathsf{tan}}

\DeclareMathOperator{\sfdet}{\mathsf{det}}
\DeclareMathOperator{\sfadj}{\mathsf{adj}}

\DeclareMathOperator{\sftr}{\mathsf{tr}}

\newcommand{\mat}[1]{\left[\: \begin{matrix} #1 \end{matrix} \:\right]}

\newcommand{\spliteq}[1]{\begin{split} #1 \end{split}}
\newcommand{\simode}[1]{\left\{\:  \begin{aligned} #1 \end{aligned} \right.}

\begin{document}

\title{
Voltage Stability Kernel: A Cofactor Theory of Voltage Stability in Lossy Power Systems
}

\author{Takayuki Ishizaki, Jigen Koizumi, Hiroo Yashiba, Boqiang Sun$^1$
\thanks{$^1$ Institute of Science Tokyo, 2-12-1, Ookayama, Meguro, Tokyo, 152-8552, Japan}
}

\IEEEtitleabstractindextext{%
\begin{abstract}
This paper introduces the voltage stability kernel (VSK), a cofactor-based bus-wise representation of voltage stability in lossy power systems.
The VSK is defined as the vector of principal cofactors of the voltage stability Laplacian (VSL), a reduced Jacobian that retains voltage source internal angles while eliminating the other variables.
We show that the VSK constitutes the left kernel of the VSL, which is typically nonsymmetric in lossy power systems.
We also define the voltage stability margin (VSM) as the sum of all VSK components and show that it is equal to the product of all eigenvalues of the VSL except the trivial zero eigenvalue due to phase-shift symmetry.
Thus, the VSK provides a bus-wise decomposition of the VSM.
Furthermore, the VSK offers an algebraic interpretation of CPF calculations with a fixed slack bus.
The singularity of the Jacobian in CPF calculations obtained by deleting the slack-bus row and column is characterized by the vanishing of the VSK component selected by the slack bus. 
In contrast, the static bifurcation is characterized by the vanishing of the VSM.
Since these two conditions are generally different, our theory explains why a CPF nose point does not necessarily correspond to a static bifurcation in lossy cases.
\end{abstract}

}

\maketitle

\section{Introduction}

Voltage stability is a fundamental issue in power system analysis and operation.
It is closely related to the solvability of power flow equations, the disappearance of feasible equilibria, and the loss of local stability through static bifurcations.
Classical studies have shown that voltage collapse can often be interpreted as a saddle-node bifurcation, where a stable equilibrium and an unstable equilibrium coalesce and disappear.
This viewpoint provides a mathematical basis for voltage stability analysis \cite{DobsonChiang1989,ChiangDobsonThomas1990,Kundur1994,VanCutsemVournas1998}.

Continuation power flow (CPF) is a widely used computational method for estimating voltage stability limits.
CPF tracks an equilibrium branch under a specified loading direction and detects the nose points of the corresponding curves \cite{AjjarapuChristy1992}.
In lossless power systems, these nose points are naturally associated with static bifurcations of the underlying dynamics.
In lossy systems, however, this correspondence is unclear.
Indeed, power-flow equations can have multiple solutions, and some low-voltage or nonstandard solutions may be stable under certain operating conditions  \cite{TamuraMoriIwamoto1983,NguyenTuritsyn2014}.
These observations suggest that a nose point of a CPF curve does not mathematically characterize a static bifurcation in lossy power systems on its own.
Thus, an algebraic criterion is needed to distinguish CPF nose points from static bifurcations of the underlying dynamics.

A key difficulty is that the relevant Jacobian for voltage instability is generally nonsymmetric in lossy power systems.
In lossless or symmetric cases, phase-shift symmetry yields a Laplacian-like matrix whose left and right kernels coincide.
However, with transmission losses, the reduced Jacobian retains a uniform phase-shift direction as the right kernel, while the left kernel becomes nonuniform.
Left eigenvectors have been used in voltage stability analysis.
For example, they represent normal directions of stability boundaries \cite{dobson1992observations}.
Furthermore, modal analysis evaluates bus participation in critical voltage modes using the left and right eigenvectors associated with small nontrivial eigenvalues of a reduced Jacobian \cite{gao1992voltage}.
In contrast, this paper focuses on the algebraic structure associated with the trivial zero eigenvalue arising from phase-shift symmetry and relates it to nontrivial static degeneracy.

This paper presents a cofactor theory of voltage stability in lossy power systems.
Starting from a standard differential-algebraic equation (DAE) model, we create a reduced Jacobian that retains the voltage source internal angles.
We call this reduced Jacobian the voltage stability Laplacian (VSL) because, despite being generally nonsymmetric, it preserves the phase-shift symmetry of the right kernel.
The reduction is based on the idea that static bifurcations can be characterized by stationary power flow variables after eliminating locally solvable internal and algebraic variables.
This reduction procedure essentially relies on the Schur complement, which preserves Laplacian-like network structures  \cite{dorfler2013kron}.

The cofactor viewpoint is inspired by classical graph theory.
For an undirected graph Laplacian, Kirchhoff’s matrix-tree theorem states that the weighted sum of spanning trees equals any principal cofactor \cite{Biggs1993}.
A nonsymmetric generalization is also shown in \cite{Chaiken1982}.
In this paper, we apply the cofactor principle to the VSL, which is obtained via the Schur complement of the power system Jacobian.
Thus, the VSL can be viewed as a voltage-stability analogue of graph-theoretic cofactor invariants.

The main contribution of this paper is introducing the voltage stability kernel (VSK).
The VSK is defined as the vector of principal cofactors of the VSL.
We prove that this vector constitutes the left kernel of the VSL.
This result provides a cofactor-theoretic counterpart to the standard Laplacian kernel property, which is particularly significant in lossy power systems where the VSL is nonsymmetric and the left and right kernels differ.

We also define the voltage stability margin (VSM) as the sum of all VSK components.
We prove that the VSM is equal to the product of all eigenvalues of the VSL except for the trivial zero eigenvalue associated with a uniform phase shift.
Thus, the VSK provides a bus-wise decomposition of the system-wide VSM.
In symmetric or lossless cases, all VSK components are identical, meaning that the bus-wise distinction disappears.
In lossy power systems, however, the VSK components can be nonuniform and can even have different signs.

Finally, we introduce an implication for CPF calculations through the VSK.
In a CPF calculation with a fixed slack bus, fixing the phase reference removes one phase degree of freedom.
The singularity of the Jacobian obtained by deleting the slack-bus row and column is characterized by the vanishing of the VSK component selected by the slack bus.
In contrast, the static bifurcation is characterized by the vanishing of the VSM.
Thus, our cofactor theory shows that these two points coincide in lossless systems but can generally differ in lossy systems.

The remainder of this paper is organized as follows.
Section~2 formulates the power system model and derives the reduced Jacobian used for static bifurcation analysis.
Section~3 develops the VSL, VSK, and VSM theory, and interprets CPF calculations through the resulting cofactor structure.
Section~4 presents numerical examples.
Section~5 concludes the paper.


\section{Problem Formulation}

This section formulates the DAE power system model and the static bifurcation problem considered in this paper.
We further explain how the internal generator states and bus voltage magnitudes can be reduced from the Jacobian.
Throughout the paper, we denote 
the set of real numbers by $\mathbb{R}$, 
the one-dimensional torus by $\mathbb{S}$, 
the $n$-dimensional all-ones vector by $\mathds{1}_n$, 
the $n$-dimensional identity matrix by $I_n$,  and
the cardinality of a set $\mathds{N}$ by $|\mathds{N}|$.

\subsection{Power System Model}

\subsubsection{Transmission Network Model}\label{sec:DAEmodel}

Let $\mathds{N}$ denote the label set of buses. 
The network admittance matrix is denoted by
\begin{equation}\label{eq:adY}
\bm{Y} = G + \bm{j}B
\end{equation}
where $G\in \mathbb{R}^{|\mathds{N}| \times |\mathds{N}|}$ and $B\in \mathbb{R}^{|\mathds{N}| \times |\mathds{N}|}$ are the conductance and susceptance matrices, respectively.
The transmission network is said to be ``lossless" if $G$ is zero, and ``lossy" otherwise.
The complex voltage phasor at bus $i$ is given by
\[
\bm{V}_i = e^{\rho_i + \bm{j}\theta_i},
\]
where $\rho_i\in \mathbb{R}$ and $\theta_i \in \mathbb{S}$ denote the logarithmic voltage magnitude and phase angle, respectively.
Note that 
\[
|\bm{V}_i|=e^{\rho_i}, \quad 
\angle \bm{V}_i =\theta_i.
\]
The active and reactive power injections at bus $i$, i.e., the power balance equations, are given by
\begin{equation}\label{eq:PQ_def}
\simode{
P_i
&=
\sum_{j=1}^{|\mathds{N}|}
e^{\rho_i + \rho_j }
\left\{
G_{ij}\sfcos(\theta_i \!-\! \theta_j)
+
B_{ij}\sfsin(\theta_i\!-\!\theta_j)
\right\}
\\
Q_i
&=
\sum_{j=1}^{|\mathds{N}|}
e^{\rho_i + \rho_j }
\left\{
G_{ij}\sfsin(\theta_i\!-\!\theta_j)
-
B_{ij}\sfcos(\theta_i\!-\!\theta_j)
\right\}
}
\end{equation}
where $G_{ij}$ and $B_{ij}$ denote the $(i,j)$-elements of $G$ and $B$, respectively.
We represent the bus variables as
\[
v=(\theta,\rho),\quad
w=(P,Q),
\]
where the symbols without the subscript $i$ represent the vectors consisting of all corresponding symbols.
Then, the power balance equation in \eqref{eq:PQ_def} can be simply written as
\begin{equation}\label{eq:PQ_def2}
w= g(v).
\end{equation}
In the following, we denote a stationary power flow distribution by
\[
\varpi^{\star}:=(v^{\star},w^{\star}),
\]
which satisfies
\[
w^{\star}= g(v^{\star}).
\]

\subsubsection{Synchronous Generator Model}

For the sake of simplicity, this paper uses a two-axis synchronous generator model for its discussion. 
Note that the same conclusions can be obtained using either the more detailed Park model or the simplified classical model for the synchronous generator \cite{nishino2025equilibrium,NishinoOnishiIshizaki2025}.

Consider a synchronous generator at bus $i$.
Let $E_{{\rm q}i}\in \mathbb{R}$ denote the field-winding flux linkage, 
$E_{{\rm d}i}\in \mathbb{R}$ denote the damper-winding flux linkage, 
$\delta_i\in \mathbb{S}$ denote the rotor angle relative to the frame rotating at the system angular frequency $\omega_0$, and 
$\omega_i \in \mathbb{R}$ be the angular frequency deviation relative to $\omega_0$.
Then, the generator dynamics is given as
\begin{subequations}\label{eq:2axis_model}
\begin{equation}\label{eq:2axis}
\simode{
\dot{\delta}_i &= \omega_0 \omega_i \\
M_i  \dot{ \omega}_i &= -D_i \omega_i - P_i + P^{\star}_{\mathrm{m}i} \\
\tau_{\mathrm{d}i} \dot{E}_{\mathrm{q}i} &= - E_{\mathrm{q}i}
 - (X_{\mathrm{d}i} - X'_{\mathrm{d}i}) I_{\mathrm{d}i} + V^{\star}_{\mathrm{fd}i} \\
\tau_{\mathrm{q}i} \dot{E}_{\mathrm{d}i} &= - E_{\mathrm{d}i}  + (X_{\mathrm{q}i} - X'_{\mathrm{q}i}) I_{\mathrm{q}i}
}
\end{equation}
where
$M_i$ is the inertia constant,
$D_i$ is the damping coefficient, 
$\tau_{{\rm d}i} $ and $\tau_{{\rm q}i} $ are the time constants of the flux linkage dynamics,
$X_{{\rm d}i}$ and $X_{{\rm q}i} $ are the d-axis and q-axis synchronous reactances,
$X_{{\rm d}i}'$ and $X_{{\rm q}i}'$ are the d-axis and q-axis transient reactances,
$P_{{\rm m}i}^{\star}$ is the mechanical input, and 
$V_{{\rm fd}i}^{\star}$ is the field voltage.
The currents $I_{{\rm d}i}$ and $I_{{\rm q}i}$ flowing into the bus along the d-axis and q-axis are given by 
\begin{equation}\label{eq:2axis_output_I}
    I_{\mathrm{d}i} = \frac{1}{X'_{\mathrm{d}i}} ( E_{\mathrm{q}i} - V_{\mathrm{q}i} ), \quad
    I_{\mathrm{q}i} = \frac{1}{X'_{\mathrm{q}i}} ( V_{\mathrm{d}i} - E_{\mathrm{d}i} )
\end{equation}
where $V_{\mathrm{d}i}$ and $V_{\mathrm{q}i}$ are defined as
\[
V_{{\rm d}i} :=
e^{\rho_i}  \sfsin (\delta_i - \theta_i ) , \quad
V_{{\rm q}i} :=
e^{\rho_i} \sfcos (\delta_i - \theta_i ) .
\]
The active power and the reactive power outputs of the synchronous generator are given as
\begin{equation}\label{eq:PiQis}
P_i  =  V_{{\rm q}i} I_{{\rm q}i} +
V_{{\rm d}i} I_{{\rm d}i} ,\quad
Q_i  =  
V_{{\rm q}i} I_{{\rm d}i} -
V_{{\rm d}i} I_{{\rm q}i}.
\end{equation}
\end{subequations}

Choosing the bus variables $(v_i,w_i)$ as the input and output for the connection to the bus, \eqref{eq:2axis_model} can be formally expressed as
\begin{equation}\label{eq:formal_model}
\simode{
\dot{x}_i &= f_i (x_i, v_i; u_i^{\star}) \\
w_i &= h_i (x_i, v_i)
}
\end{equation}
where $x_i$ is the state variable, and $u_i^{\star}$ is a constant input representing the mechanical power and field voltage.  

It is known that the stationary state $x_i^{\star}$ and the constant input $u_i^{\star}$ such that
\[
0 = f_i (x_i^{\star}, v_i^{\star}; u_i^{\star}) 
,\quad
w_i^{\star} = h_i (x_i^{\star}; v_i^{\star})
\]
are uniquely determined for a given stationary power flow $\varpi_{i}^{\star}$ at bus $i$. 
In particular, the phase difference $\delta^{\star}_i - \theta^{\star}_i$ between the generator and bus at the stationary power flow is uniquely determined as
\begin{equation}\label{eq:def_phi}
\phi_i(\nu_i^{\star}) := \sftan^{-1} \biggl( \frac{P_i^{\star}}{Q_i^{\star} + 
\frac{e^{2\rho_i^{\star}}}{X_{{\rm q}i}}} \biggr)
\end{equation}
where the set of the stationary voltage, and active and reactive powers is denoted by
\[
\nu_i^{\star}:=(\rho_i^{\star},P_i^{\star},Q_i^{\star}).
\]
From this fact, the compatible value of $u_i^{\star}$ can be written as
\begin{equation}\label{eq:inputst}
\spliteq{
P_{{\rm m}i}^{\star} & = P^{\star}_i, \\
V_{{\rm fd}i}^{\star} & = \! \tfrac{ X_{{\rm d}i}P_i^{\star} }{e^{\rho_i^{\star}} } \sfsin \phi_i(\nu_i^{\star})  \!+ \!
\left( \! \tfrac{ X_{{\rm d}i}Q_i^{\star} }{e^{\rho_i^{\star}}}  + e^{\rho_i^{\star}}  \! \right) \sfcos \phi_i(\nu_i^{\star}).
}
\end{equation}
We assume that the mechanical power and field voltage are set to satisfy \eqref{eq:inputst} for a given stationary power flow.

\subsubsection{Constant-Power Load Model}
The active and reactive power injections at bus $i$ are specified by constants
$P_{{\rm c}i}^{\star}$ and  $Q_{{\rm c}i}^{\star}$.
Under the constant-power model, these power injections and consumptions are independent of the bus voltage magnitude and phase. 
Equivalently, the active and reactive power injections at bus $i$ are given by
\begin{equation}\label{eq:conPQ}
P_i = P_{{\rm c}i}^{\star},
\quad
Q_i = Q_{{\rm c}i}^{\star}.
\end{equation}
Note that the negative values of $P_{{\rm c}i}^{\star}$ and $Q_{{\rm c}i}^{\star}$ represent consumption.
The constant power load model in \eqref{eq:conPQ} can be formally expressed as
\begin{equation}\label{eq:formal_model2}
w_i = u_i^{\star},
\quad \forall v_i \in \mathbb{S}\times \mathbb{R}
\end{equation}
where $u_i^{\star}$ is a constant input representing the active and reactive power injections.

For simplicity, this paper focuses on constant-power loads. 
However, the subsequent cofactor-based stability analysis is not limited to this model.
It can be applied to constant-current, constant-impedance, and composite load models as well.
Only the corresponding Jacobian blocks require modification.

\subsubsection{Power System Model}
The entire power system model is obtained as a nonlinear DAE composed of
\begin{subequations}\label{eq:formal_DAE}
\begin{equation}
\simode{
\dot{x} &= f (x,v_{\mathds G};u_{\mathds G}^{\star}) \\
w_{\mathds G} &= h_{\mathds G}(x,v_{\mathds G}) 
}
\end{equation}
for the generator buses, the set of which is denoted by $\mathds{G}$, 
\begin{equation}
w_{\mathds L} = h_{\mathds L}(u_{\mathds L}^{\star}) 
\end{equation}
for the load buses,  the set of which is denoted by $\mathds{L}$, and
\begin{equation}
\simode{
w_{\mathds G}&=g_{\mathds G}(v_{\mathds G},v_{\mathds L})\\
w_{\mathds L}&=g_{\mathds L}(v_{\mathds G},v_{\mathds L}).
}
\end{equation}
\end{subequations}
Without loss of generality, we assume that 
\[
\mathds{N} = \mathds{G} \cup \mathds{L},
\quad
\mathds{G} \cap \mathds{L} = \emptyset.
\]
Note that a stationary state $x^{\star}$ is uniquely determined for each stationary power flow distribution $\varpi^{\star}$.
In the following, we denote an equilibrium by $e^{\star}:=(x^{\star},\varpi^{\star})$.

\subsection{Static Bifurcation in Power Systems}

\subsubsection{Local Asymptotic Stability of An Equilibrium}

We rewrite \eqref{eq:formal_DAE} in a compact form as
\begin{equation}\label{eq:DAEcom}
\simode{
\dot{x} &= f (x,v;u^{\star}) \\
0 &= h(x,v;u^{\star}) -g(v)
}
\end{equation}
where $h$ and $g$ denote the stacked compositions of $h_{\mathds{G}},h_{\mathds{L}}$ and $g_{\mathds{G}},g_{\mathds{L}}$, respectively.
For a given equilibrium $e^{\star}$, we analyze its local asymptotic stability via linearization.
The linearized version of \eqref{eq:DAEcom} is obtained as
\[
\mat{
\Delta \dot{x} \\
0
}
=
\mat{
\frac{\partial f}{\partial x} (e^{\star}) & \frac{\partial f}{\partial v} (e^{\star}) \\
\frac{\partial h}{\partial x} (e^{\star}) & \frac{\partial h}{\partial v} (e^{\star}) -\frac{\partial g}{\partial v} (e^{\star})
}
\mat{\Delta x\\ \Delta v}
\]
where $\Delta x$ and $\Delta v$ represent the deviations of $x$ and $v$ from their stationary values $x^{\star}$ and $v^{\star}$.
Eliminating the algebraic variable $\Delta v$, we obtain the equivalent representation 
\begin{equation}\label{eq:genA}
\Delta \dot{x} = \underbrace{\left\{
\frac{\partial f}{\partial x}  - \frac{\partial f}{\partial v} 
\left(
\frac{\partial h}{\partial v}  -\frac{\partial g}{\partial v} 
\right)^{-1}
\frac{\partial h}{\partial x} 
\right\}
}_{A(e^{\star}) }
\Delta x
\end{equation}
in an ordinary differential equation (ODE) form.

Note that due to the uniform phase-shift symmetry of the power system, the equilibrium is not isolated. 
Rather, it belongs to an equivalence class generated by uniformly shifting all generator internal angles and bus voltage phases.
Accordingly, local asymptotic stability is understood modulo this phase-shift symmetry. 
In the linearized ODE representation, this symmetry gives rise to a trivial zero eigenvalue of $A(e^{\star}) $ associated with the uniform phase shift. 
Excluding this trivial zero eigenvalue, local stability can be lost with respect to a system parameter when an additional real eigenvalue reaches the origin or when a pair of complex conjugate eigenvalues crosses the imaginary axis. 
The former corresponds to a static bifurcation, and the latter corresponds to a Hopf bifurcation. 
This paper focuses on static bifurcations, the type of instability underlying the voltage stability limits.

\subsubsection{Motivating Example of Continuation Power Flow}\label{sec:CPFex}

\begin{figure}[t]
\centering
\includegraphics[width = .90\linewidth]{ 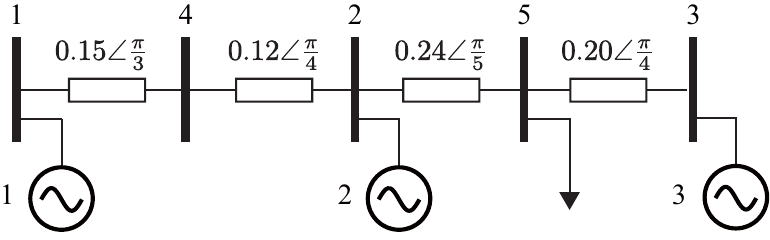}
\medskip
\caption{Example of 5-bus power system.
}
\label{fig:expsys}
\end{figure}

\begin{table}[t]
\caption{Generator constants.}\label{table:genpara}
\centering
\begin{tabular}{cccccccccccccccc}
Gen &  $M$  &  $D$ &  $\tau_{\rm d}$ & $\tau_{\rm q}$ &  $X_{\rm d}$ & $X_{\rm q}$ & $X_{\rm d}'$ & $X_{\rm q}'$  \\
\hline
1 &  $ 12$ &  2  &  8.97  &  1.50 &  0.09 &  0.09  &  0.029 &  0.029  \\
2 &  $ 18$ &  2  &  5.90  &  1.50 &  0.12 &  0.12  &  0.036 &  0.036  \\
3 &  $ 30$ &  3  &  5.14  &  1.50 &  0.20 &  0.20  &  0.062 &  0.062   \\  \hline
\end{tabular}
\end{table}

\begin{table}[t]
\caption{CPF parameter setting.}\label{table:CPFset}
\centering
\begin{tabular}{lccccccccccccccc}
Bus &  $P$  &  $Q$ & $V_{\rm fd}$ \\
\hline
1 (slack) &           &          & 1.03 \\
2 (PV) &  $0.30\lambda$  &             & 1.02 \\
3 (PV) &  $0.50\lambda$  &            & 1.06 \\
4 (PQ) &  0       &  0        & \\
5 (PQ) &  $-1.00\lambda$ &  $0.20\lambda$  &    & \\  \hline
\end{tabular}
\end{table}

This subsection uses a standard CPF calculation as an example of static bifurcation analysis.
The CPF tracks an equilibrium branch under a specified loading direction and is commonly used to estimate voltage stability limits.
For demonstration, we consider the 5-bus power system example composed of three generators and one load shown in Fig.~\ref{fig:expsys}, where the impedances of the transmission lines are shown.
The generator constants are listed in Table~\ref{table:genpara}.
The system angular frequency $\omega_0$ is $120\pi$.

In a CPF calculation, one generator is designated as the slack bus and the others are designated as PV buses to find an equilibrium.
Specifically, the internal angle and field voltage are specified at the slack bus, and the active power and field voltage are specified at each PV bus.
In this example, we consider the case in which the first generator is designated as the slack bus.
The other two generators are designated as PV buses.

On the other hand, buses 4 and 5 are designated as PQ buses, at which the active and reactive power injections are specified.
In this example, only bus 5 has a non-zero power injection due to the constant-power load.
Power injection to bus 4 is zero.
The parameter setting of the CPF calculation is summarized as in Table~\ref{table:CPFset}, where $\lambda$  is the load factor.
A large value of the load factor indicates large power transmission.

\begin{figure}[t]
\centering
\includegraphics[width = .9\linewidth]{ 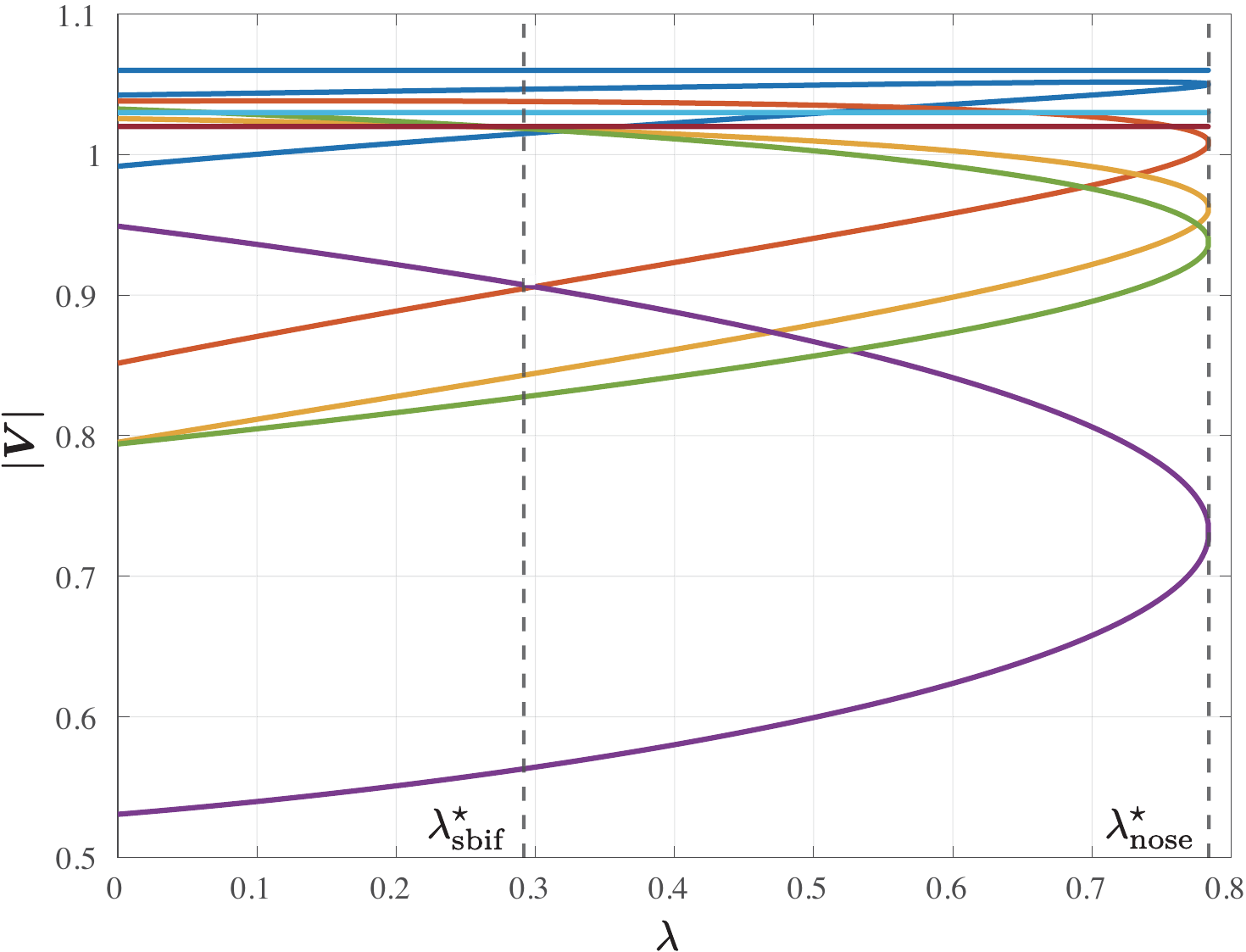}
\medskip
\caption{
Nose curves obtained by CPF calculation of lossy power systems.
}
\label{fig:lossynose}
\end{figure}

\begin{figure}[t]
\centering
\includegraphics[width = .9\linewidth]{ 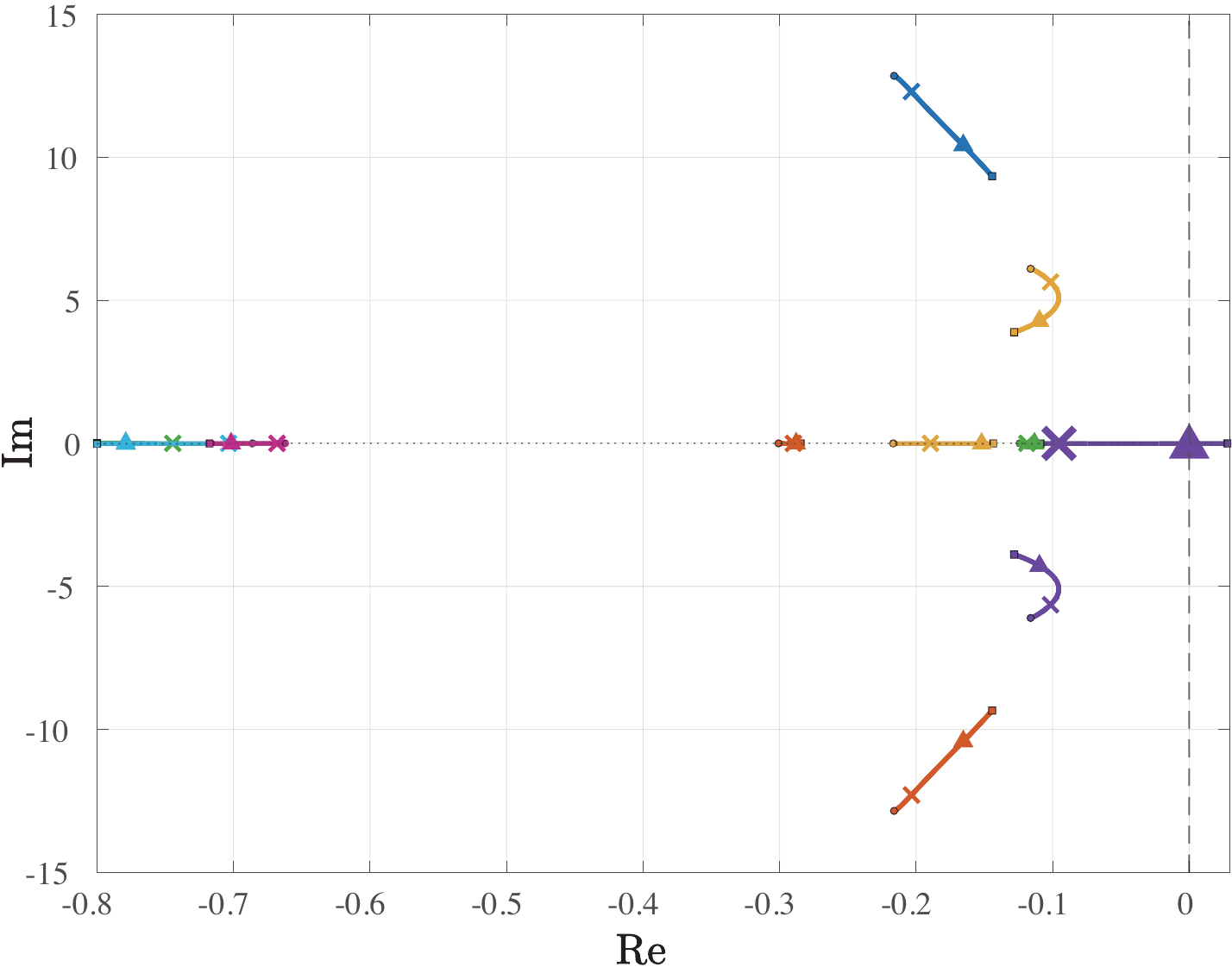}
\medskip
\caption{Change in eigenvalues of lossy power systems with 2-axis generator models.
The cross marks indicate the critical load factor of the nose point.
The triangle marks indicate that of the static bifurcation.
}
\label{fig:eig2axislossy}
\end{figure}

In general, we can find a locally asymptotically stable equilibrium for a sufficiently small load factor.
The CPF method tracks changes in this equilibrium branch as the load factor is varied continuously or incrementally.
The branch reaches a nose point at a critical load factor, beyond which this branch cannot be continued in the loading direction.
The nose curves of buses 1--5 and generators are shown in Fig.~\ref{fig:lossynose}.
The load factor of the nose points is 
\begin{equation}\label{eq:lamnose}
\lambda^{\star}_{\rm nose} \simeq 0.784. 
\end{equation}
Fig.~\ref{fig:eig2axislossy} shows a plot of the change in the eigenvalues of $A(e^{\star})$ in \eqref{eq:genA} for each equilibrium obtained by the CPF calculation.
From this figure, we can see that increasing the load factor causes a static bifurcation.
Along the lower-voltage branch of the same CPF curve, a real eigenvalue reaches the origin.
The load factor of this static bifurcation is 
\begin{equation}\label{eq:lamsbif}
\lambda^{\star}_{\rm sbif} \simeq 0.292.
\end{equation}
In fact, the critical load factors of the nose point and static bifurcation do not coincide.
The cross and triangle marks in Fig.~\ref{fig:eig2axislossy} indicate the corresponding eigenvalues.
This paper will explain the mathematical basis for this gap.

\subsection{Reduction for Static Bifurcation Analysis}

\subsubsection{Reduction of Generator States}

This subsection explains that each synchronous generator can be equivalently replaced with a simpler model in the static bifurcation analysis.
For this purpose, we consider a simple voltage source model, which can also be interpreted as a frequency droop control inverter model, given as 
\begin{subequations}\label{eq:droop_state_DAE}
\begin{equation}
\tfrac{D_i}{\omega_0} \dot{\delta}_i = - P_i + P^{\star}_{\mathrm{m}i}
\end{equation}
with the d-axis and q-axis current equations
\begin{equation}\label{eq:classical_output_I}
    I_{\mathrm{d}i} = \frac{1}{X_{\mathrm{d}i}} ( V^{\star}_{\mathrm{fd}i} - V_{\mathrm{q}i} ) 
    ,\quad
    I_{\mathrm{q}i} = \frac{1}{X_{\mathrm{q}i}} V_{\mathrm{d}i}.
\end{equation}
The active and reactive power outputs are given as
\begin{equation}
P_i  =  V_{{\rm q}i} I_{{\rm q}i} +
V_{{\rm d}i} I_{{\rm d}i} ,\quad
Q_i  =  
V_{{\rm q}i} I_{{\rm d}i} -
V_{{\rm d}i} I_{{\rm q}i},
\end{equation}
\end{subequations}
which are identical to those of the synchronous generator.
Similarly to \eqref{eq:2axis_model}, this model can be formally expressed as
\begin{equation}\label{eq:formal_model_fdc}
\simode{
\dot{\delta}_i &= \hat{f}_i (\delta_i, v_i; u_i^{\star}) \\
w_i &= \hat{h}_i (\delta_i, v_i; u_i^{\star}).
}
\end{equation}
The power system model where all generators are replaced with the simple voltage sources can also be represented as
\begin{equation}\label{eq:DAEcom2}
\simode{
\dot{\delta} &= \hat{f} (\delta,v;u^{\star}) \\
0 &= \hat{h}(\delta,v;u^{\star}) -g(v).
}
\end{equation}
The corresponding linearized ODE model is obtained as
\begin{equation}\label{eq:lin_fdc}
\Delta \dot{\delta} = \underbrace{\left\{
\frac{\partial \hat{f}}{\partial \delta}  - \frac{\partial \hat{f}}{\partial v} 
\left(
\frac{\partial \hat{h}}{\partial v} -\frac{\partial g}{\partial v} 
\right)^{-1}
\frac{\partial \hat{h}}{\partial \delta} 
\right\}
}_{\hat{A}(e^{\star}) }
\Delta \delta.
\end{equation}

\begin{figure}[t]
\centering
\includegraphics[width = .9\linewidth]{ 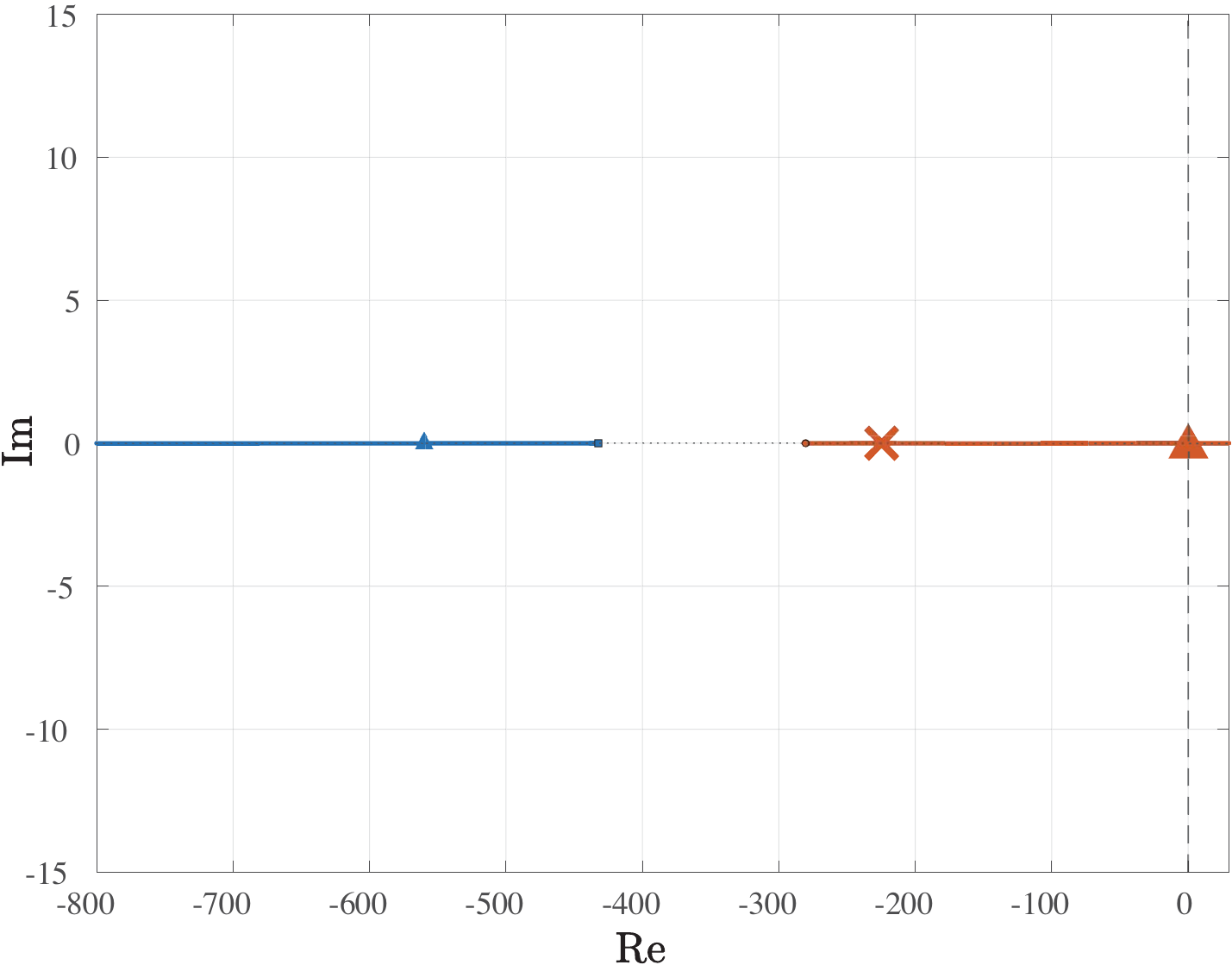}
\medskip
\caption{Change in eigenvalues of lossy power systems with simple voltage source models.
The cross marks indicate the critical load factor of the nose point.
The triangle marks indicate that of the static bifurcation.
}
\label{fig:eigdrooplossy}
\end{figure}

Fig.~\ref{fig:eigdrooplossy} shows a plot of the change in the eigenvalues of $\hat{A}(e^{\star})$ in \eqref{eq:lin_fdc} obtained from the same CPF analysis as in Section~\ref{sec:CPFex}.
In fact, though the eigenvalues differ significantly, the nontrivial real eigenvalue reaches the origin at the same critical load factor $\lambda^{\star}_{\rm sbif}$ in \eqref{eq:lamsbif}.
This result can be mathematically proven using the implicit function theorem with an analysis of the generator energy function as follows.

Consider the stationary equation of \eqref{eq:2axis_model} with respect to the variables $\omega_i$, $E_{{\rm q}i}$, and $E_{{\rm d}i}$ as
\[
\simode{
0 &= -D_i \omega_i - P_i + P^{\star}_{\mathrm{m}i} \\
0 &= - E_{\mathrm{q}i}
 - (X_{\mathrm{d}i} - X'_{\mathrm{d}i}) I_{\mathrm{d}i} + V^{\star}_{\mathrm{fd}i} \\
0 &= - E_{\mathrm{d}i}  + (X_{\mathrm{q}i} - X'_{\mathrm{q}i}) I_{\mathrm{q}i}.
}
\]
Eliminating $(\omega_i,E_{{\rm q}i},E_{{\rm d}i})$ from the differential equation of $\delta_i$, and the equations of the d-axis current $I_{{\rm d}i}$ and the q-axis current $I_{{\rm q}i}$ leads to the simplified voltage source model in \eqref{eq:droop_state_DAE}.
This means that the same stationary value of $\delta_i^{\star}$ is found by the CPF analysis for both the synchronous generator model in \eqref{eq:2axis_model} and the simplified voltage source model in \eqref{eq:droop_state_DAE} with the same $V^{\star}_{\mathrm{fd}i}$.

In fact, this reduction operation with respect to the stationary values is invertible around any feasible equilibrium.
This fact follows from the implicit function theorem, since the partial Jacobian with respect to $(\omega_i,E_{{\rm q}i},E_{{\rm d}i})$ is always nonsingular. 
The nonsingularity of the partial Jacobian is proven by the fact that the corresponding partial Hessian of the generator energy function is always positive definite; see \cite{NishinoOnishiIshizaki2025} for details.
Hence, the internal generator dynamics does not introduce additional degeneracy.
This means that the static bifurcation relevant to voltage stability can be equivalently analyzed using the retained variables.

\subsubsection{Reduction of Bus Voltage Variables}

We use the gradient-based representation of the simplified voltage source model as
\[
\simode{
\tfrac{D_i}{\omega_0}\dot{\delta}_i&= - \tfrac{\partial U_i}{\partial \delta_i} (\delta_i,v_i) + P^{\star}_{\mathrm{m}i} \\
w_i &= -\tfrac{\partial U_i}{\partial v_i}(\delta_i,v_i) 
}
\]
where $U_i$ is the energy function defined as
\begin{equation} \label{eq:droop_potential}
U_i (\delta_i,v_i)
= \frac{V^2_{\mathrm{d}i}}{2 X_{\mathrm{q}i}} +
\frac{( V^{\star}_{\mathrm{fd}i} - V_{\mathrm{q}i} )^2 }{2 X_{\mathrm{d}i}}.
\end{equation}
Let $U$ be the stacked composition of $U_i$.
Then, the linearized DAE can be written as
\[
\mat{D \Delta \dot{\delta} \\ 0}
= - 
\Biggl(
\underbrace{
\mat{
\frac{\partial^2 U}{\partial \delta^2}   & \frac{\partial^2 U}{\partial \delta \partial v} \\
\frac{\partial^2 U}{\partial v \partial \delta} & \frac{\partial^2 U}{\partial v^2} 
}
+
\mat{
0& 0 \\
0 & \frac{\partial g}{\partial v}
}
}_{J(e^{\star})}
\Biggr)
\mat{\Delta \delta \\ \Delta v}
\]
where $D$ is the positive diagonal matrix composed of $D_i/\omega_0$.
The first term of $J(e^{\star})$ is relevant to the voltage source, and it is the Hessian of the energy function being symmetric.
The second term is relevant to the transmission network, and it is generally not symmetric in lossy cases.

We denote the block matrices of $J(e^{\star})$ by
\begin{equation}\label{eq:blkJ}
J(e^{\star}) =\mat{
J_{\delta \delta}(e^{\star})  & J_{\delta v}(e^{\star}) \\
J_{v \delta}(e^{\star}) & J_{vv}(e^{\star})
}.
\end{equation}
Then, $\hat{A}(e^{\star})$ in \eqref{eq:lin_fdc} can be represented as
\begin{equation}\label{eq:hatA_L}
\hat{A}(e^{\star}) = - D^{-1} L(e^{\star})
\end{equation}
where $L(e^{\star})$ is the reduced Jacobian  of $J(e^{\star})$ defined as
\begin{equation}\label{eq:VSL}
L(e^{\star}):=
J_{\delta \delta}(e^{\star}) 
- J_{\delta v}(e^{\star}) J_{vv}^{-1}(e^{\star}) J_{v \delta}(e^{\star}) .
\end{equation}
We will use $L(e^{\star})$ as a key indicator of static bifurcations.

Throughout this paper, we assume that the partial Jacobian $J_{vv}(e^{\star})$ is nonsingular at the equilibria under consideration.
This assumption is reasonable from a physical standpoint because the bus voltage variables must be uniquely determined around the equilibrium.
It is important to note that the Schur complement operation in \eqref{eq:VSL} does not ignore the effect of the bus voltage variables. 
Rather, their stationary characteristics are retained in a mathematically equivalent manner.

\section{Voltage Stability Kernel Theory}

This section develops the cofactor theory of voltage stability based on the reduced Jacobian in the previous section.
We then interpret CPF calculations through this cofactor structure, highlighting why CPF nose points can differ from static bifurcations in lossy power systems.

\subsection{Voltage Stability Laplacian, Kernel, and Margin}

\subsubsection{General Case of Lossy Power Systems}

The purpose of this subsection is to extract algebraic information from the reduced Jacobian $L(e^\star)$ in \eqref{eq:VSL}.
For simplicity, we first assume that the time constant matrix $D$ in \eqref{eq:hatA_L} is the identity matrix.
In Section~\ref{sec:gentime}, we will provide a generalization to the case of nonuniform time constants.
In the following, a ``bus" refers to a voltage source bus.
We first introduce the following terminology. 

\begin{definition}
The reduced Jacobian $L(e^{\star})$ defined as in \eqref{eq:VSL} is called a \textbf{voltage stability Laplacian} (VSL).
\end{definition}

From the definition, it is clear that a nontrivial degeneracy of the VSL is equivalent to the static bifurcation under consideration.
Although the VSL is generally nonsymmetric in lossy power systems, it retains a Laplacian-like structure induced by the uniform phase-shift symmetry. 
This will be proven as follows.

\begin{lemma}\label{lem:VSL}
For $L(e^{\star})$ in \eqref{eq:VSL}, it follows that
\begin{equation}\label{eq:Laps}
L(e^{\star}) \mathds{1}_{|\mathds{G}|} =0.
\end{equation}
\end{lemma}

\begin{IEEEproof}
Due to the uniform phase-shift symmetry, it follows that
\[
\underbrace{
\mat{
J_{\delta \delta} & J_{\delta \theta} & J_{\delta \rho} \\
J_{\theta \delta} & J_{\theta \theta} & J_{\theta \rho} \\
J_{\rho \delta} & J_{\rho \theta} & J_{\rho \rho} 
}
}_{J}
\mat{
\mathds{1}_{|\mathds{G}|} \\ \mathds{1}_{|\mathds{N}|}\\0
}
=0
.
\]
Therefore, we have
\[
J_{\delta \delta} \mathds{1}_{|\mathds{G}|}
+
J_{\delta v}
q_v
=0,\quad
J_{v \delta}
\mathds{1}_{|\mathds{G}|}
+
J_{vv} q_v=0
\]
where the stacked vector and matrices are defined as
\[
q_v=\mat{
\mathds{1}_{|\mathds{N}|}\\0
}
,\quad
J_{\delta v} = \mat{J_{\delta \theta} & J_{\delta \rho}}
,\quad
J_{v \delta}=\mat{
J_{\theta \delta} \\ J_{\rho \delta}
}.
\]
Thus, eliminating $q$ leads to \eqref{eq:Laps}.
\end{IEEEproof}
\smallskip

Lemma~\ref{lem:VSL} shows that the VSL always has a trivial zero eigenvalue. 
Therefore, the ordinary determinant of the VSL is identically zero and cannot be used directly as a bifurcation indicator. 
Instead, we use the following principal cofactors, which remain informative even in the presence of this trivial zero eigenvalue.

\begin{definition}
Let $L_{-i}(e^{\star})$ denote the submatrix obtained by removing the $i$th row and column from $L(e^{\star})$ in \eqref{eq:VSL}.
Then
\begin{equation}\label{eq:VSK}
\kappa(e^{\star}) := \mat{
\sfdet ( L_{-1}(e^{\star}) ) \\
\vdots \\
\sfdet ( L_{-|\mathds{G}|}(e^{\star}) )
}
\end{equation}
is called a \textbf{voltage stability kernel} (VSK).
\end{definition}

The determinant of $L_{-i}(e^{\star})$, the $i$th element of the VSK, corresponds to the $i$th diagonal element of the cofactor of the VSL, meaning that the VSK is the vector of principal cofactors.
The terminology ``kernel'' is justified as follows.

\begin{lemma}\label{lem:VSK}
For $L(e^{\star})$ in \eqref{eq:VSL}, it follows that
\begin{equation}
\kappa^{\sf T}(e^{\star}) L(e^{\star})  =0
\end{equation}
where $\kappa(e^{\star})$ is defined as in \eqref{eq:VSK}.
\end{lemma}

\begin{IEEEproof}
Let $\sfadj(L)$ denote the adjugate matrix of $L$.
The adjugate matrix satisfies the identity
\[
\sfadj(L)L
=
L\sfadj(L)
=
\sfdet(L)I_{|\mathds{G}|}.
\]
Note that $\sfdet(L)$ is zero due to the trivial zero eigenvalue.
Thus, we have
\begin{equation}\label{eq:reladj1}
\sfadj(L)L=0.
\end{equation}
We next show that
\[
\sfadj(L)
=
\mathds{1}_{|\mathds{G}|}\kappa^{\sf T}.
\]
Let $C$ denote the cofactor matrix of $L$, whose $(i,j)$-entry is given by
\[
C_{ij}
=
(-1)^{i+j}\sfdet\left(L_{-ij}\right),
\]
where $L_{-ij}$ denotes the submatrix obtained by deleting the $i$th row and the $j$th column of $L$.
By definition of the adjugate matrix, we have
\[
\sfadj(L)=C^{\sf T}.
\]
We show the structure of $C$.
Let $R_i$ be the matrix obtained by deleting the $i$th row of $L$.
Denote the columns of $R_i$ by
\[
R_i = \mat{r_1 & \cdots & r_{|\mathds{G}|}}.
\]
As shown in \eqref{eq:Laps}, the sum of all columns of $L$ is zero.
Therefore, the sum of all columns of $R_i$ is also zero as
\[
r_1+\cdots+r_{|\mathds{G}|}=0.
\]
This implies that, for any $j\in\mathds{G}$
\[
r_j
=
-\sum_{k \in\mathds{G}\setminus\{j\}} r_k .
\]
Using the multilinearity and alternating property of the determinant, it follows that
\[
(-1)^1\sfdet\left(L_{-i1}\right)
=\cdots =
(-1)^{|\mathds{G}|}
\sfdet \left(L_{-i|\mathds{G}|}\right).
\]
Hence, the cofactors in the $i$th row of $C$ satisfy
\[
C_{ij}=C_{ii},
\quad
\forall j\in\mathds{G}.
\]
From the definition of $\kappa_i$, the $i$th element of $\kappa$, we have
\[
C_{ii}
=
\kappa_i.
\]
Thus, the cofactor matrix has the structure
\[
C
=
\mat{
\kappa_1 & \cdots & \kappa_1\\
\vdots & \ddots & \vdots \\
\kappa_{|\mathds{G}|} & \cdots & \kappa_{|\mathds{G}|}
}.
\]
Consequently, we have
\begin{equation}\label{eq:reladj2}
\sfadj(L)
=
C^{\sf T}
=
\mathds{1}_{|\mathds{G}|}\kappa^{\sf T}.
\end{equation}
Combining \eqref{eq:reladj1} and \eqref{eq:reladj2} proves the claim.
\end{IEEEproof}
\smallskip

Lemma~\ref{lem:VSK} states that the vector of principal cofactors is not merely a collection of minors, but forms a left kernel of the VSL. 
This is particularly important in lossy power systems, where the VSL is generally nonsymmetric, meaning that the left and right kernels need not coincide.

The primary contribution of this paper is to show that the VSK quantifies a voltage stability margin of  lossy power systems on a bus-by-bus basis. 
To this end, we introduce the following metric.

\begin{definition}
Let $\kappa_i(e^{\star})$ denote the $i$th component of $\kappa(e^{\star})$ in \eqref{eq:VSK}.
Then
\begin{equation}\label{eq:VSM}
\kappa_{{\rm tot}}(e^{\star}):=\sum_{i=1}^{|\mathds{G}|} \kappa_i (e^{\star})
\end{equation}
is called a \textbf{voltage stability margin} (VSM).
\end{definition}

The VSM is defined as the total contribution of all VSK components.
The following theorem shows that this cofactor-based quantity has a spectral interpretation as the product of all nontrivial eigenvalues of the VSL.

\begin{theorem}\label{thm:VSKthm}
For $\kappa_{{\rm tot}}(e^{\star})$ in \eqref{eq:VSM}, it follows that
\begin{equation}\label{eq:eqmu}
\kappa_{{\rm tot}}(e^{\star}) = \prod_{\lambda \in {\bm \Lambda}_L (e^{\star})} \lambda
\end{equation}
where ${\bm \Lambda}_L (e^{\star})$ denotes the multiset of all the eigenvalues of $L(e^{\star})$ in \eqref{eq:VSL} including algebraic multiplicity except the trivial zero eigenvalue.
\end{theorem}

\begin{IEEEproof}
Consider the scalar polynomial
\[
p(\epsilon):=\sfdet(L+\epsilon I_{|\mathds{G}|}).
\]
Since the eigenvalues of $L+\epsilon I_{|\mathds{G}|}$ are shifted by $\epsilon$, we have
\[
p(\epsilon)
=
\epsilon \prod_{\lambda \in {\bm \Lambda}_L } (\lambda +\epsilon).
\]
Therefore, we have
\[
\frac{d p}{d\epsilon}(0)
=
\prod_{\lambda \in {\bm \Lambda}_L } \lambda.
\]

On the other hand, using an identity of the derivative of the determinant, we have
\[
\frac{d p}{d\epsilon}(\epsilon)
=
\sftr\left(
\sfadj(L+\epsilon I_{|\mathds{G}|})
\right).
\]
Evaluating this identity at zero gives
\[
\frac{d p}{d\epsilon}(0)
=
\sftr(\sfadj(L)) =
\sum_{i=1}^{|\mathds{G}|}\kappa_i
=\kappa_{\rm tot}
\]
where we have used \eqref{eq:reladj2} to derive the second equality.
This proves \eqref{eq:eqmu}.
\end{IEEEproof}
\smallskip

Theorem~\ref{thm:VSKthm} shows that the VSM is equivalently characterized by the ``pseudo-determinant" of the VSL, defined as the product of all eigenvalues of the VSL except the trivial zero eigenvalue.
Thus, the vanishing of the VSM 
\begin{equation}\label{eq:stbifcon}
\kappa_{{\rm tot}}(e^{\star}) = 0
\end{equation}
detects a nontrivial degeneracy of the VSL, or equivalently a static bifurcation.
 
The key message of Theorem~\ref{thm:VSKthm} is that in a lossy power system with a nonsymmetric VSL, the contribution of each bus to system-wide voltage stability is not uniform because the VSK components are not uniform.
Note that VSK components can take on negative values in a lossy power system. 
A negative VSK component can be interpreted as a signed contribution that offsets the positive contributions of other buses in the cofactor decomposition of the VSM.
The system-wide VSL degenerates when the sum of the voltage stability contributions of all voltage sources is zero. 
Therefore, the vanishing of a particular VSK component does not necessarily imply a VSL degeneracy.

\subsubsection{Special Case of Lossless Power Systems}

In a general lossy power system, the VSK components are not necessarily identical because the VSL is generally nonsymmetric.
On the other hand, in lossless power systems, the VSL becomes symmetric under the present formulation.
Motivated by this fact, we state the following result in symmetric cases.

\begin{theorem}\label{thm:VSK2}
Suppose that $L(e^{\star})$ in \eqref{eq:VSL} is symmetric. 
Then, for $\kappa_i(e^{\star})$, the $i$th component of $\kappa(e^{\star})$ in \eqref{eq:VSK}, it follows that
\begin{equation}\label{eq:eqmull}
\kappa_i (e^{\star}) = \frac{1 }{|\mathds{G}|} \kappa_{{\rm tot}}(e^{\star}) ,
\quad
\forall i \in \mathds{G}
\end{equation}
where $\kappa_{{\rm tot}}(e^{\star})$ is defined as in \eqref{eq:VSM}.
\end{theorem}

\begin{IEEEproof}
Due to the symmetry of $L$, its left and right kernels coincide.
Thus, there exists a scalar $\alpha$ such that
\[
\kappa = \alpha \mathds{1}_{|\mathds{G}|}.
\]
From \eqref{eq:eqmu}, $\alpha$ is found as $\tfrac{\kappa_{{\rm tot}}}{|\mathds{G}|} $.
This proves the claim.
\end{IEEEproof}
\smallskip

Theorem~\ref{thm:VSK2} shows that the bus-wise distinction of the VSK disappears when the VSL is symmetric.
In this case, all principal cofactors are identical.
Therefore, the vanishing of one VSK component is equivalent to the vanishing of all VSK components, as well as that of the VSM.


\subsubsection{Case of Nonuniform Time Constants}\label{sec:gentime}

We consider the case of nonuniform time constants, in which $D$ in \eqref{eq:hatA_L} is not the identity matrix.
Define
\[
L^{D}(e^{\star}):=D^{-1} L(e^{\star}).
\]
Because $\mathds{1}_{|\mathds{G}|}$ forms its right kernel, the notions of the VSL, VSK, and VSM can also be generalized to $L^{D}(e^{\star})$.
In particular, the VSK of $L^{D}(e^{\star})$ is obtained as
\begin{equation}\label{eq:mVSK}
\kappa^D (e^{\star}):= \frac{1}{\sfdet(D)} D \kappa (e^{\star})
\end{equation}
where $\kappa (e^{\star})$ denotes the VSK of $L(e^{\star})$.
Similarly, the VSM is modified as
\begin{equation}\label{eq:mVSM}
\kappa^D_{\rm tot} (e^{\star}):= \frac{1}{\sfdet(D)} \sum_{i=1}^{|\mathds{G}|} d_i \kappa_i (e^{\star})
\end{equation}
where $d_i$ is the $i$th diagonal element of $D$.
The vanishing of the modified VSM detects a nontrivial degeneracy of the modified VSL in lossy cases.
It is noteworthy that, despite $D$ being nonsingular, the nontrivial degeneracy of $L(e^{\star})$ and $L^{D}(e^{\star})$ does not necessarily coincide.
In contrast, the vanishing of one VSK component is equivalent to the vanishing of all VSK components and the VSM in lossless cases.
Therefore, the nontrivial degeneracy coincides.

\subsection{Implication to Continuation Power Flow}

We now revisit the CPF calculation discussed in Section~\ref{sec:CPFex} from the viewpoint of the VSK.
The purpose of this subsection is to provide an algebraic interpretation of the nose points in CPF calculations.

Suppose that the slack bus is labeled by $s\in\mathds{G}$.
The reduced Jacobian associated with the fixed slack bus is represented as the principal submatrix $J_{-s}(e^\star)$ obtained by deleting the $s$th row and column from $J(e^\star)$, which is the full Jacobian in the  DAE form.
Therefore, the loss of local regularity with the fixed slack bus is characterized by
\[
\sfdet(J_{-s}(e^\star))=0.
\]
In fact, this is equivalent to
\begin{equation}\label{eq:nosecon}
\kappa_s(e^\star)=0,
\end{equation}
which can be proven using the determinant identity with respect to the Schur complement operation as
\[
\sfdet(J_{-s}(e^\star)) = \sfdet(J_{vv}(e^\star)) 
\underbrace{
\sfdet(L_{-s}(e^\star))
}_{\kappa_s(e^\star)}
,
\]
where the partial Jacobian $J_{vv}(e^\star)$ is nonsingular at the equilibria under consideration.

This observation reveals the algebraic condition that underlies the nose point of a CPF calculation with a fixed slack bus.
The loss of local regularity of the Jacobian obtained by deleting the slack-bus row and column is characterized by \eqref{eq:nosecon}.
This condition differs from the static bifurcation condition, which is characterized by \eqref{eq:stbifcon}.

In lossy power systems, the VSL is generally nonsymmetric and the VSK components are nonuniform.
Therefore, the vanishing of a particular VSK component does not necessarily coincide with the vanishing of the VSM.
Consequently, the nose point of a CPF calculation with a fixed slack bus does not necessarily correspond to a static bifurcation.
In contrast, when the VSL is symmetric, Theorem~\ref{thm:VSK2} implies that all VSK components are identical.
With nonuniform time constants, the modified VSK components are scaled by the positive diagonal elements of $D$, as shown in Section~\ref{sec:gentime}.
Hence, the critical load factors of the CPF nose points and the static bifurcation still coincide in symmetric or lossless cases.

\begin{figure}[t]
\centering
\includegraphics[width = .9\linewidth]{ 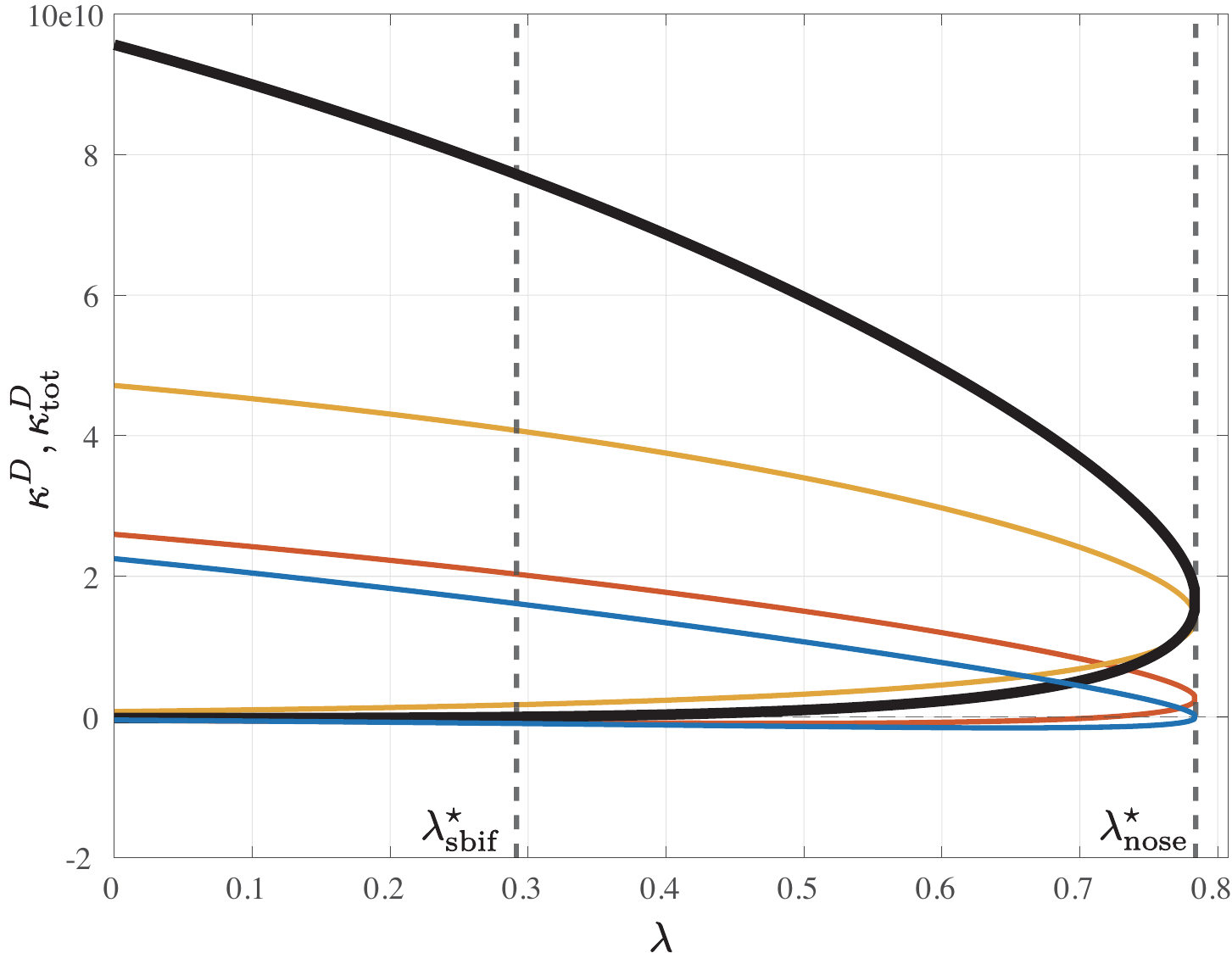}
\medskip
\caption{
Change in VSK components (thin lines) and VSM (black thick line) of lossy power systems.
}
\label{fig:VSKlossy}
\end{figure}

\begin{figure}[t]
\centering
\includegraphics[width = .9\linewidth]{ 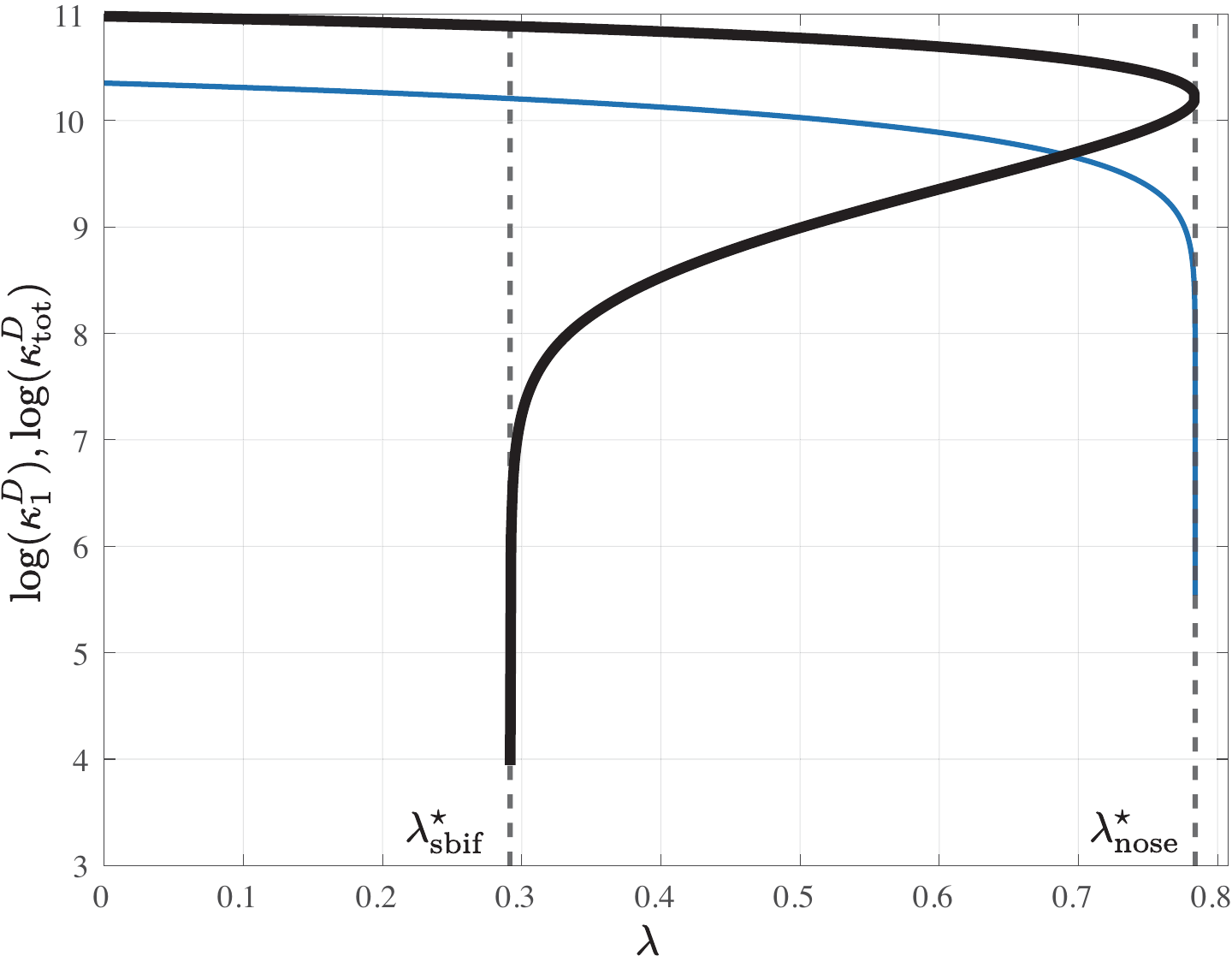}
\medskip
\caption{
Change in VSK component of the slack bus (blue thin line) and VSM (black thick line) of lossy power systems.
}
\label{fig:VSMlossy}
\end{figure}

\section{Numerical Examples}

\subsection{Case of Lossy Power Systems}

We use the same example in Section~\ref{sec:CPFex}.
Consider the modified versions of the VSK and VSM in \eqref{eq:mVSK} and \eqref{eq:mVSM} because the time constants are not uniform in this example.
Fig.~\ref{fig:VSKlossy} shows a plot of the change in the VSK components for each equilibrium obtained by
the CPF calculation.
We can see that the VSK components are not uniform and may have negative values.
Fig.~\ref{fig:VSMlossy} shows the logarithm of the first VSK component, corresponding to the slack bus, and that of the VSM.
The plot shows only the range where those values are positive.
We can see that the vanishing of the first VSK component occurs at the critical load factor $\lambda_{\rm nose}^{\star}$  in \eqref{eq:lamnose}, and that of the VSM occurs at $\lambda_{\rm sbif}^{\star}$  in \eqref{eq:lamsbif}.
These results are consistent with our theory.

\subsection{Case of Lossless Power Systems}

As an example of lossless cases, we consider the case where the phase angle of all transmission line impedances is set to $\pi/2$ while their absolute values remain unchanged.
All other parameter settings are the same.

The results are shown in Figs.~\ref{fig:losslessnose}--\ref{fig:VSMlossless}.
As our theory proves, the vanishing of one VSK component is equivalent to the vanishing of all VSK components and the VSM.
Therefore, the static bifurcation occurs at the critical load factor of the nose points in the CPF calculation.
For reference, the critical load factor is
\[
\lambda_{\rm nose}^{\star} = \lambda_{\rm sbif}^{\star} \simeq 2.605.
\]

\begin{figure}[t]
\centering
\includegraphics[width = .9\linewidth]{ 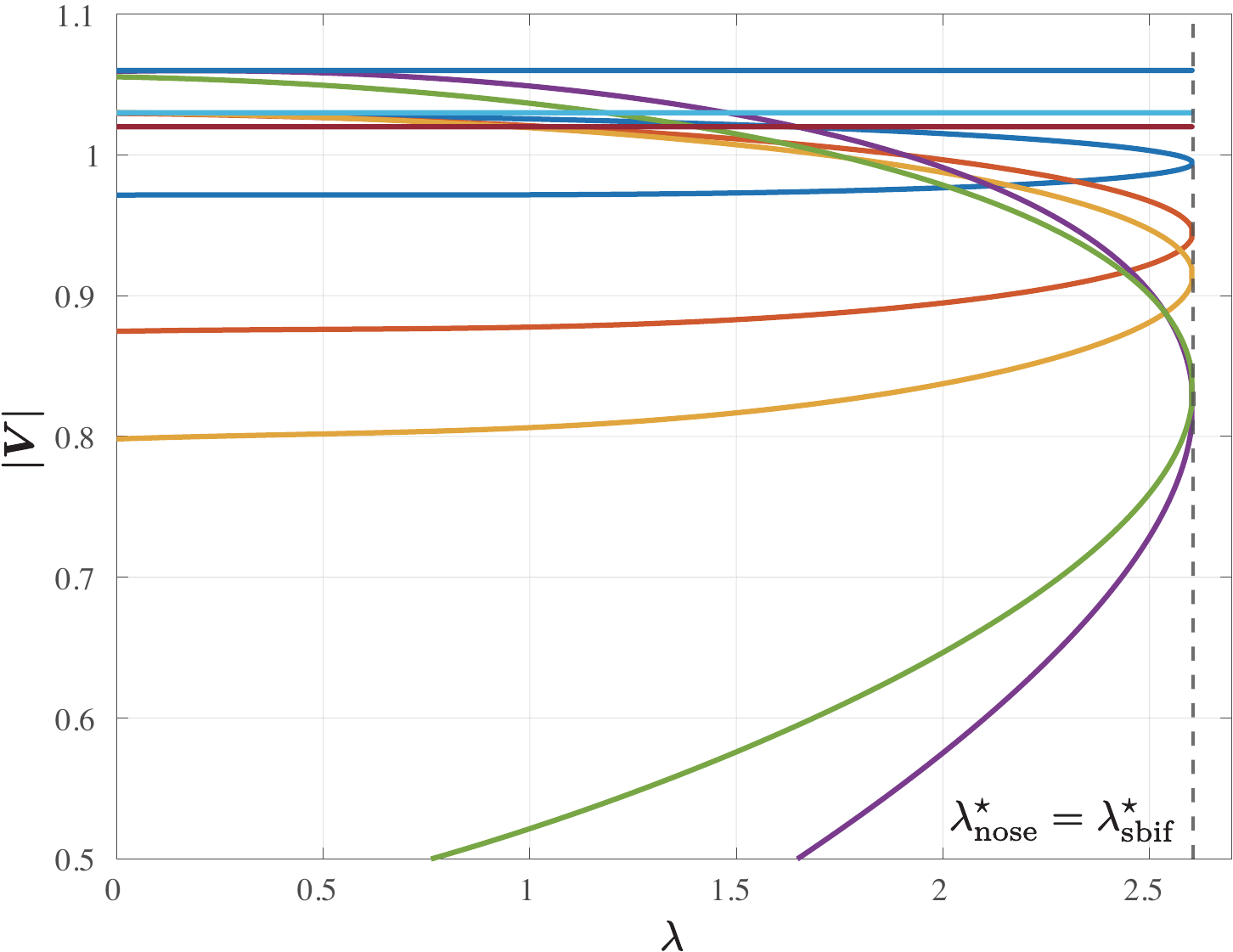}
\medskip
\caption{Nose curves obtained by CPF calculation of lossless power systems.
}
\label{fig:losslessnose}
\end{figure}

\begin{figure}[t]
\centering
\includegraphics[width = .9\linewidth]{ 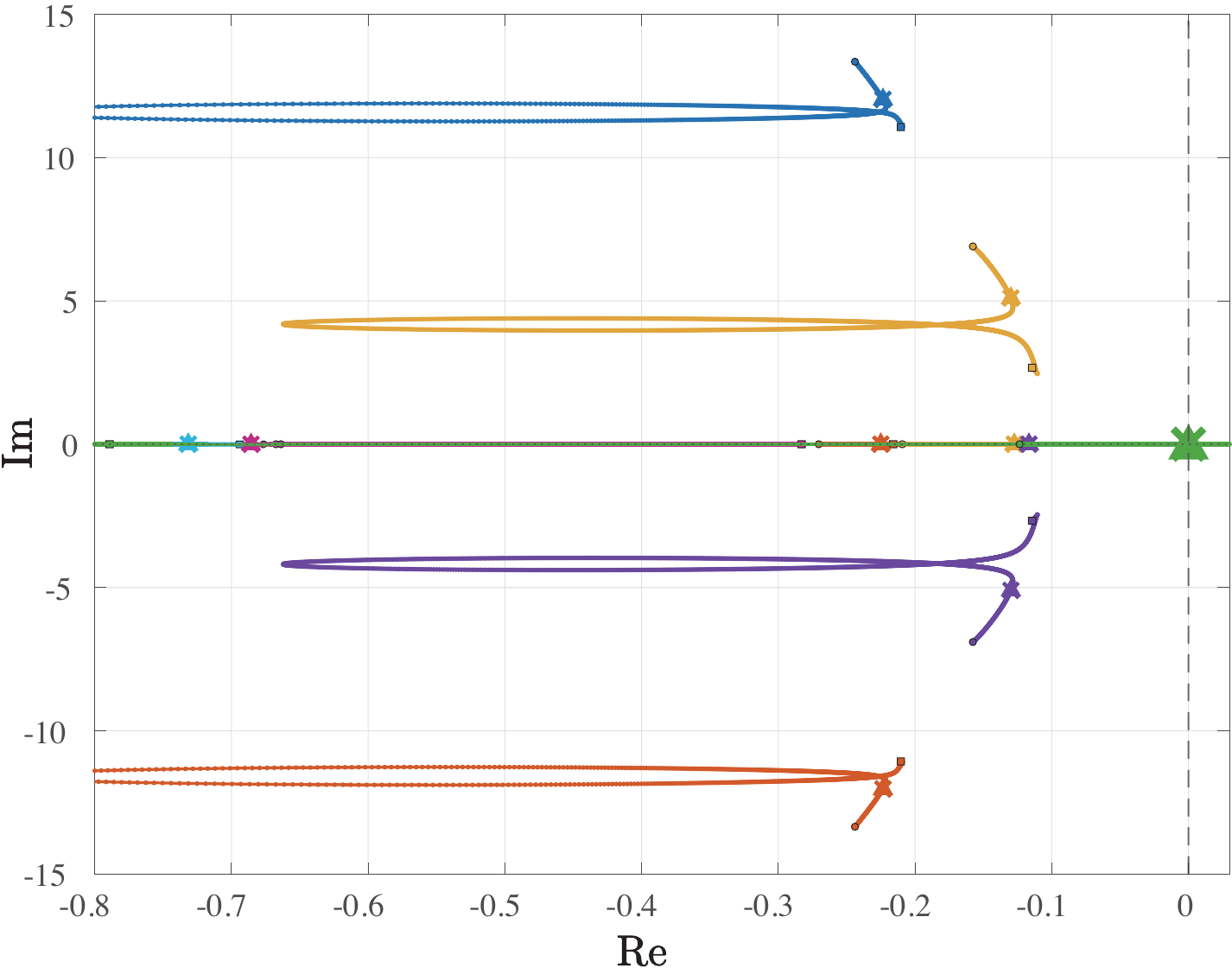}
\medskip
\caption{Change in eigenvalues of lossless power systems with 2-axis generator models.
The cross marks indicate the critical load factor of the nose point.
The triangle marks indicate that of the static bifurcation.
}
\label{fig:eig2axislossless}
\end{figure}

\begin{figure}[t]
\centering
\includegraphics[width = .9\linewidth]{ 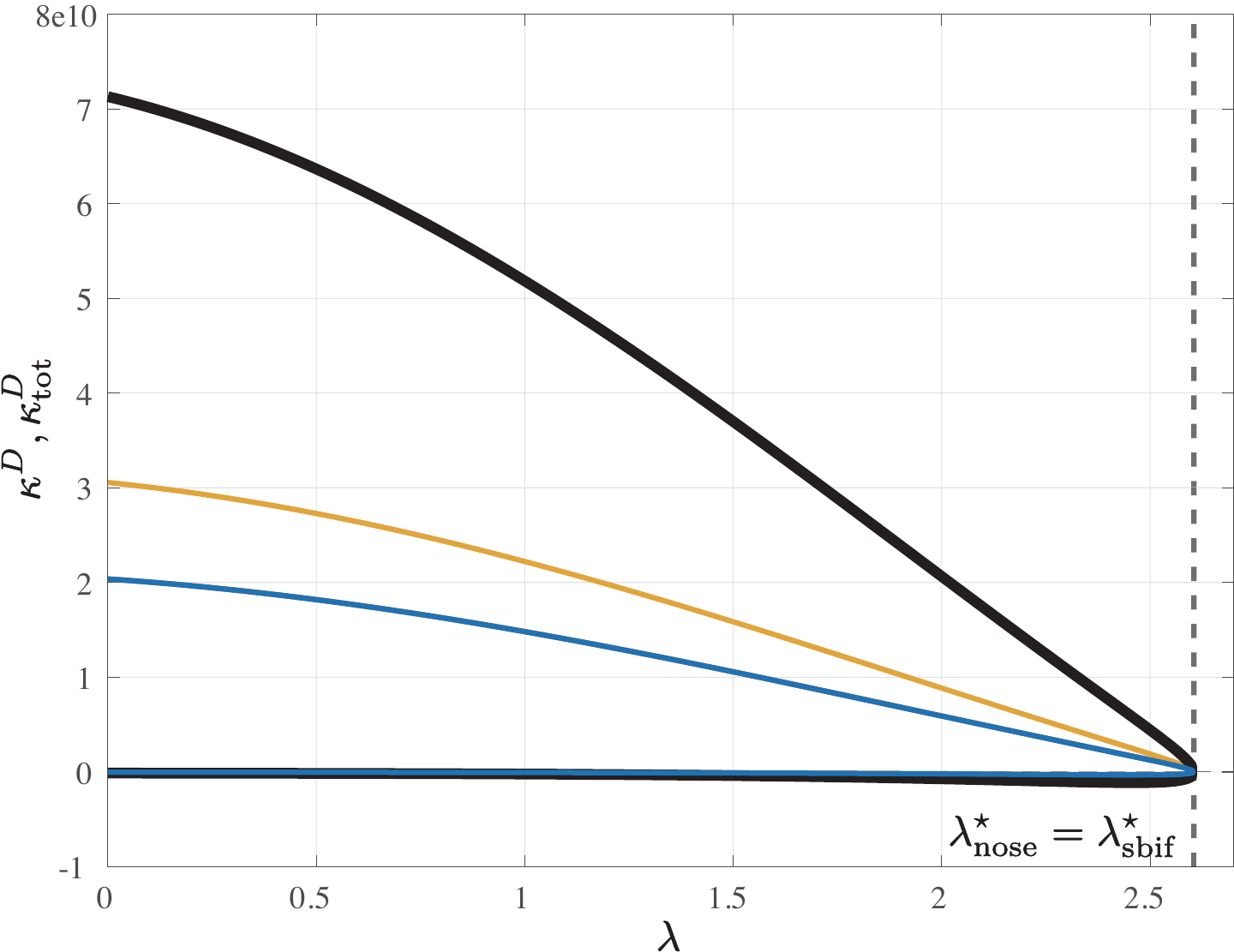}
\medskip
\caption{
Change in VSK components (thin lines) and VSM (black thick line) of lossless power systems.
}
\label{fig:VSKlossless}
\end{figure}

\begin{figure}[t]
\centering
\includegraphics[width = .9\linewidth]{ 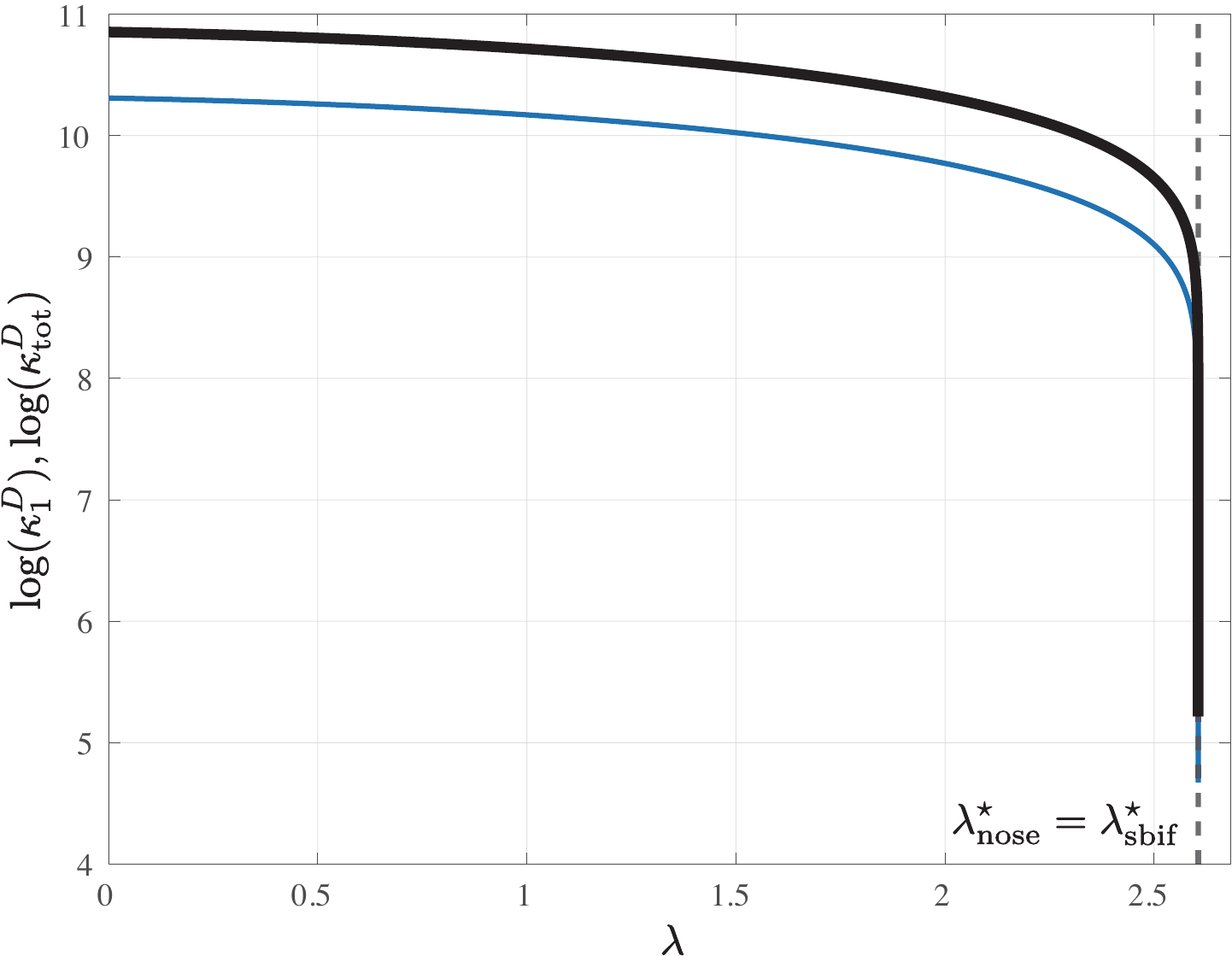}
\medskip
\caption{
Change in VSK component of the slack bus (blue thin line) and VSM (black thick line) of lossless power systems.
}
\label{fig:VSMlossless}
\end{figure}

\section{Concluding Remarks}

This paper introduced the VSK, a cofactor-based representation of voltage stability in lossy power systems.
The VSK is defined as the vector of principal cofactors of the VSL, a reduced Jacobian that retains voltage source internal angles.
We showed that the VSK forms the left kernel of the generally nonsymmetric VSL and that the VSM is decomposed into the sum of all VSK components.
This cofactor structure offers a bus-wise interpretation of voltage stability and explains why a CPF nose point does not necessarily correspond to a static bifurcation in lossy power systems.

Future work will include extending the cofactor theory to hierarchical model reductions.
Specifically, the distribution of the VSM components should be preserved under Schur complement operations.
Similarly, the VSM should have a precise relationship with the unreduced Jacobian through determinant identities.
Using these properties would enable tracking of both bus-wise and system-wide voltage stability information across different levels of reduced models.

The proposed theory is also relevant to next-generation power systems with widespread grid-forming (GFM) voltage sources.
As GFM inverters are deployed in low-voltage layers, voltage formation will not be concentrated only at conventional high-voltage layers.
In this context, the VSK could provide a mathematical foundation for analyzing multi-layer voltage formation and support across transmission, subtransmission, and distribution systems.

\section*{Acknowledgment}

This paper is based on results obtained from a project, JPNP24007, commissioned by the New Energy and Industrial Technology Development Organization (NEDO).

\ifCLASSOPTIONcaptionsoff
  \newpage
\fi

\bibliographystyle{unsrt}
\bibliography{refs}

\end{document}